\documentclass{amsart}
\usepackage{amscd, amssymb, amsmath, amsthm}
\usepackage{amsfonts,enumerate}
\usepackage{latexsym,,txfonts}
\usepackage{graphics,graphicx,psfrag}
\usepackage[colorlinks=true,citebordercolor={0 1 0}]{hyperref}

\newtheorem{theorem}{Theorem}[section]
\newtheorem{lemma}[theorem]{Lemma}

\theoremstyle{definition}

\theoremstyle{remark}

\begin{document}

\title{Ribbon concordance of knots is a partial ordering}

\author[Ian Agol]{%
        Ian Agol} 
\address{%
    University of California, Berkeley \\
    970 Evans Hall \#3840 \\
    Berkeley, CA 94720-3840} 
\email{%
     ianagol@math.berkeley.edu}

\thanks{Ian Agol is supported by a Simons Investigator grant}

\date{%
January 8, 2022}

\begin{abstract}
In this note we show that ribbon concordance forms a partial ordering on the set of knots, answering a question of Gordon \cite[Conjecture 1.1]{Gordon81}. The proof makes use of representation varieties of the knot groups to $SO(N)$ and relations between them induced by a ribbon concordance. 
 
\end{abstract}

\maketitle


\section{Introduction}

\begin{center}
\begin{figure}[htb]\includegraphics[height = 1.5in]{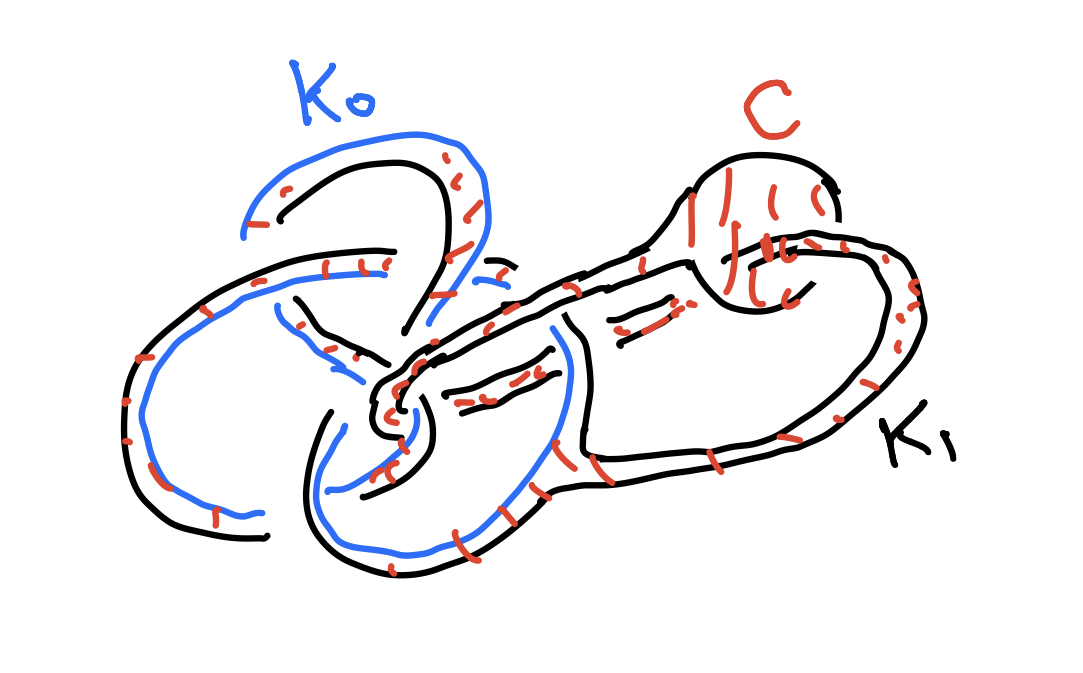}
\caption{A ribbon concordance from a knot to the figure eight knot} \label{ribbon}

\end{figure}
\end{center}

A concordance between knots $K_0, K_1 \subset S^3$ is a smooth embedded annulus $e:(S^1\times[0,1], S^1\times\{0\},S^1\times\{1\}) \to (S^3\times [0,1], K_0\times\{0\}, K_1\times\{1\})$. We also call the image of the annulus $e(S^1\times [0,1])= C\subset S^3\times[0,1]$ a concordance from $K_1$ to $K_0$. If the projection $S^3\times [0,1] \to [0,1]$ is a Morse function when restricted to $C$ with only critical points of index $0$ and $1$ (so no local maxima), then we say that $C$ is a {\it ribbon concordance} from $K_1$ to $K_0$ (introduced in \cite{Gordon81}). We write $K_1\geq K_0$ if there is a ribbon concordance from $K_1$ to $K_0$. Projecting into $S^3$, one may see $K_1$ and $K_0$ bounding an immersed annulus $C$ with ribbon singularities intersecting only $K_1$ (see Figure \ref{ribbon} for an example).

The main Conjecture 1.1 of \cite{Gordon81} states that this relation is a partial order. The ribbon concordance relation is reflexive and transitive, so the conjecture amounts to asking if it is antisymmetric? That is, if $K_1\geq K_0$ and $K_0\geq K_1$, is $K_0$ isotopic to $K_1$? Gordon answers this conjecture for knots satisfying various hypotheses, as a special case if $K_0$ or $K_1$ is fibered. Much more evidence has been amassed for this conjecture: if $K_0\geq K_1 \geq K_0$, then $K_0$ and $K_1$ have the same S-equivalence class \cite{Gilmer84}, Seifert genus and knot Floer homology \cite{Zemke19}, Khovanov homology \cite{LevineZemke19}, and instanton knot Floer homology \cite[Theorem 7.4]{KM21}. 
 
The main result of this note is to answer Gordon's conjecture positively: 
\begin{theorem}\label{order}
Ribbon concordance is a partial order.
\end{theorem}

This follows pretty immediately from the following (compare \cite[Theorem 1.4]{Gordon81}):
\begin{theorem}\label{selfconcordance}
Let $C$ be a ribbon concordance from $K$ to $K$. Then the exterior of $C$ is a relative $s$-cobordism from the exterior of $K$ to itself. 
\end{theorem}

In the conclusion we point out that this theorem potentially generalizes to homology ribbon cobordism in the sense of \cite{DLVW} and we consider the possibility of answering some other questions from \cite[Section 6]{Gordon81}.

{\bf Acknowledgement:} We thank Josh Greene for introducing us to this problem and Bernd Sturmfels and Rainer Sinn for suggestions regarding real algebraic geometry and Nathan Dunfield for comments and references. 

\section{Proof of the main theorems}

\begin{proof}[Proof of Theorem \ref{selfconcordance}]
For $N>0$, let $R_N(\pi)$ be the representation variety of the group $\pi$ to $SO(N)$. This is a real algebraic set ( the zero-set of polynomials in $\mathbb{R}^k$ for some $k$) for $\pi$ finitely generated, with coordinates given by coordinates of the matrices of the generators. Define $R_N(X)=R_N(\pi_1(X))$ for a connected manifold $X$.

We have a ribbon concordance $C\subset S^3\times [0,1]$ from $K\subset S^3\times \{0\}$ to $K\subset S^3\times \{1\}$. Let $X$ and $X'$ denote the exterior of $K$ in $S^3\times\{0\}$ and $S^3\times\{1\}$ respectively, and let $Y$ denote the exterior of $C$ in $S^3\times[0,1]$. 
By \cite[Lemma 3.1]{Gordon81}, $\iota :\pi_1(X')\to \pi_1(Y)$ is surjective (where the map is induced by inclusion), hence the induced map $R_N(Y)\to R_N(X')$ is injective. For our argument, we need to know something slightly stronger, that $R_N(Y)\subseteq R_N(X')$ is an algebraic subset. The point is here that since $\iota:\pi_1(X')\to \pi_1(Y)$ is surjective, we may take a presentation $\pi_1(X')\cong \langle g_1,\ldots, g_k | r_1,\ldots , r_{k-1} \rangle$, and use the surjection to get a presentation $\pi_1(Y)\cong \langle g_1,\ldots, g_k | r_1, \ldots, r_m \rangle$, where $r_k,\ldots , r_m$ are the extra relations that hold in $\pi_1(Y)$. Then we see that $R_N(Y)$ is an algebraic subset of $R_N(X')$, with coordinates given by the matrix coordinates of the matrices in $SO(N)$ corresponding to $g_1,\ldots, g_k$, together with relations corresponding to the relations defining $SO(N)$ for each matrix and the relators being the identity in $r_1, \ldots r_k$ or $r_1,\ldots, r_m$ respectively. 

Also by \cite[Lemma 3.1]{Gordon81} the map $\pi_1(X)\to \pi_1(Y)$ is injective. 
By \cite[Proposition 2.1]{DLVW} the induced map $R_N(Y)\to R_N(X)$ is surjective. Both of these results follow from a result of Gerstenhaber-Rothaus  \cite[Theorem 1(ii)]{GerstenhaberRothaus} which allows one to extend a representation $\rho: \pi_1(X)\to SO(N)$ to a representation $\rho': \pi_1(Y)\to SO(N)$ which restricts to $\rho$ using the fact that $Y$ has a handle decomposition with $k$ 1-handles and $k$ 2-handles added to a collar neighborhood of $X$, and so that the $2$-handles homologically cancel the $1$-handles to obtain a homology cobordism (this is called a ribbon homology cobordism in \cite{DLVW}). Note that the map $R_N(Y)\to R_N(X)$ may be with respect to different coordinates, since the generators of $\pi_1(X)$ may be regarded as a subset of the generators of $\pi_1(Y)$, and hence this polynomial map is a projection onto the subspace corresponding to the generators of $\pi_1(X)$. There is a polynomial isomorphism from $R_N(X)$ to $R_N(X')$ given by a sequence of Tietze transformations. Hence we get a surjective polynomial map $R_N(Y)\to R_N(X')$ by composing the projection $R_N(Y)\to R_N(X)$ with the polynomial isomorphism $R_N(X)\to R_N(X')$. 

We want to show that $\iota: \pi_1(X')\to \pi_1(Y)$ is injective and hence an isomorphism.  

Let $c\in \pi_1(X')-\{1\}$, and choose $N$ so that there is a representation $\rho:\pi_1(X')\to SO(N)$ such that $\rho(c)\neq 1$ (using the fact that $\pi_1(X')$ is residually finite, see \cite{Hempel87}, and that any finite group $G$ embeds into $SO(N)$ for some $N$ ). Then $R_N(Y)\subseteq R_N(X')$ is a real algebraic subset with $R_N(Y)\twoheadrightarrow R_N(X')$ a surjective polynomial map by the above discussion. 
 Then by Lemma \ref{algebraic} $R_N(Y)=R_N(X')$. 
Thus $\rho$ is the image of a representation $\rho': \pi_1(Y) \to SO(N)$. Hence $\rho'(\iota(c))\neq 1$, and therefore $\iota(c)$ is non-trivial in $\pi_1(Y)$. Thus $\iota: \pi_1(X')\to \pi_1(Y)$ is injective, and hence an isomorphism. 

The argument finishes the same as the 4th paragraph of the argument of \cite[Lemma 3.2]{Gordon81} and the proof of \cite[Theorem 1.4]{Gordon81}. We briefly summarize the fundamental group result: the map $\pi_1 (X)\to \pi_1(Y) \cong \pi_1(X')$ induced by the embedding $X\subset Y$ is injective preserving the peripheral structure, and thus induces a cover $X\to X'$ by a theorem of Waldhausen (note that there are knots such as the torus knots whose complements self-cover, but not isomorphic on the peripheral subgroup). Because the peripheral subgroup map is injective, this cover is index 1 and hence $\pi_1(X)\to \pi_1(Y) \cong \pi_1(X')$ is an isomorphism preserving the peripheral structure (meridian and longitude are sent to meridian and longitude). This implies that the knots are isotopic. The proof that $Y$ is an s-cobordism is the same as the proof of \cite[Theorem 1.4]{Gordon81} (this is not necessary for the proof of Theorem \ref{order}). 
\end{proof}

\begin{center}
\begin{figure}\includegraphics[height = 2 in]{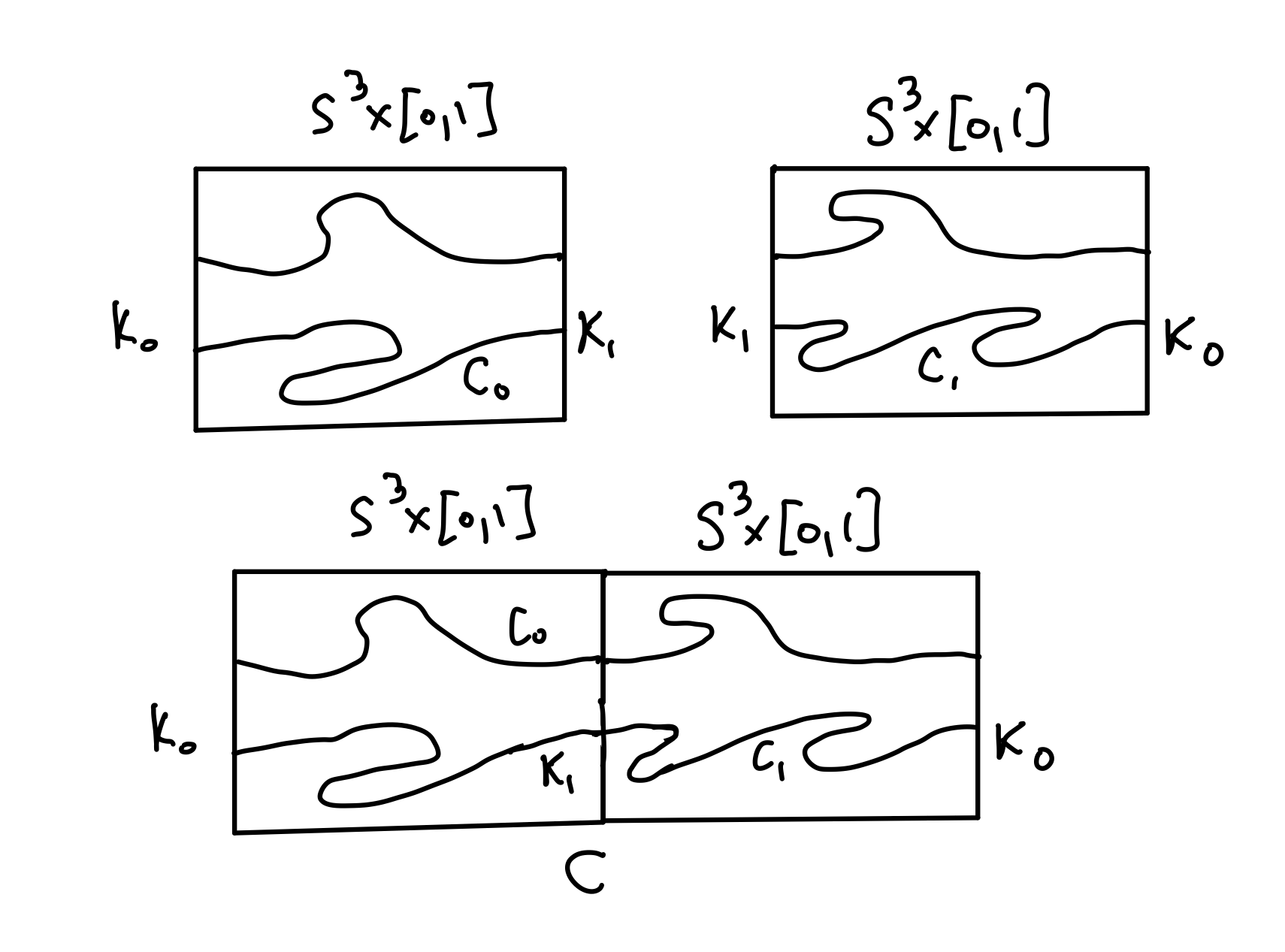}
\caption{Composing ribbon concordances to get a self-concordance} \label{concordance}

\end{figure}
\end{center}

\begin{proof}[Proof of Theorem \ref{order}]
As observed before, we only need to show that the relation is antisymmetric. 

Now, let $K_1\geq K_0$ by a ribbon concordance $C_0$ and $K_0\geq K_1$ by a ribbon concordance $C_1$.  Concatenating the ribbon concordances, we get a ribbon concordance $C=C_0C_1$ from $K_0$ to $K_0$ (see Figure \ref{concordance}). Let $Y$ be the exterior of $C$, $Y_i$ the exterior of $C_i$, $X_i$ the exterior of $K_i$. By Theorem \ref{selfconcordance}, $   \pi_1(X_0)\to \pi_1(Y)$ is an isomorphism (induced by the embedding on the right end of the concordance). Hence the map $\pi_1(X_0)\to \pi_1(Y_1)$ is also an isomorphism. Now again the argument finishes the same as the 4th paragraph of the argument of \cite[Lemma 3.2]{Gordon81} to show that $\pi_1(X_0) \cong \pi_1(X_1)$ preserving the peripheral subgroups, and hence $K_0$ and $K_1$ are isotopic. 
\end{proof}

\section{Conclusion}
We make some remarks on this argument and generalizations and the prospect for addressing some Questions from \cite[Section 6]{Gordon81}. 

The first version of the proof of Theorem \ref{order} used the fact that each knot group embeds into $SO(N)$ for some $N$ \cite{Agol18, PrzytyckiWise12}. However, we realized that this result is overkill and that we only needed residual finiteness. Theorem \ref{order} might generalize to the setting of $\mathbb{Q}$-homology ribbon cobordisms to prove \cite[Conjecture 1.1]{DLVW}, generalizing \cite[Conjecture 1.1]{Gordon81}. The proof that a self-homology ribbon cobordism has isomorphic fundamental group ought to carry over, but we do not know how to show that it is an s-cobordism. The same proof applies to knots which are {\it strongly homotopy ribbon concordant} in the sense of \cite{MillerZemke}. 

One natural question arising from the proof of Theorem \ref{order} is whether one may extract an invariant from $R_N(S^3-K)$ which preserves the partial order? A natural invariant is the ordered list of dimensions of the irreducible components of $R_N(S^3-K)$, considered up to lexicographic ordering. Then this ordering is compatible with the partial ordering of ribbon concordance of knots and thus might give an obstruction to ribbon concordance (but only in one direction for each $N$ since lexicographic order is a total order). For example, if the lexicographic ordering is reversed for two different $N$, then $K_0$ and $K_1$ could not be related by ribbon concordance in either order. 

One could hope to apply the proof of Theorem \ref{order} to answer \cite[Question 6.2]{Gordon81}. Given a sequence of knots $K_1 \geq K_2 \geq K_3 \geq \cdots$, does there exist $n$ so that $K_m=K_n$ for all $m\geq n$? The proof of Theorem \ref{order} shows that the dimensions of the representation varieties must stabilize in some sense. But to prove injectivity one would need to know that there is a faithful representation independent of $N$ which is not known in general. One special case that might work is if all the $K_i$ are hyperbolic. Assuming \cite[Conjecture 1.9]{CRS}, one would know that each knot has a faithful $SO(3)$ representation (lying on the main component of the $PSL_2(\mathbb{C})$ character variety containing the discrete faithful representation). Then the proof of Theorem \ref{order} considering $SO(3)$ representations would show that Question 6.2 holds for such a sequence. Also assuming \cite[Conjecture 1.9]{CRS}, one might be able to show that main component of the $A$-polynomial of $K_{i+1}$ divides the $A$-polynomial of $K_i$ \cite{CCGLS}, and potentially get some information about \cite[Question 6.4]{Gordon81}.

\appendix

\section{A lemma in real algebraic geometry}

Here we describe the real algebraic geometry needed to obtain the main result. We will be working in the classical setting with real algebraic sets; although we reference \cite{BCR}, we do not need to work at the level of generality of that book (only a few  definitions and results from Chapter 2). 


The following result is classical for algebraically closed fields, but we did not find the explicit statement for algebraic sets over the real numbers. 

\begin{lemma}\label{dimension}
Let $X$ be an irreducible real algebraic set, and $Y\subseteq X$ an algebraic subset. Then $dim(Y)=dim(X)$ iff $Y=X$.
\end{lemma}
\begin{proof}
We recall some notation from \cite[Chapter 2]{BCR}. 
For an algebraic set $A\subset \mathbb{R}^n$, let $\mathcal{I}(A)$ denote the polynomials in $\mathbb{R}[x_1,\ldots,x_n]$ vanishing on $A$. Let $\mathcal{P}(A)=\mathbb{R}[x_1,\ldots,x_n]/\mathcal{I}(A)$ denote the polynomial ring of $A$. 

By \cite[Definition 2.8.1]{BCR}, the dimension of $X$ is the Krull dimension of the polynomial ring $\mathcal{P}(A)$, that is, the maximal length of a chain of prime ideals of $\mathcal{P}(X)$. We may assume that $Y$ is irreducible (equivalently $\mathcal{I}(Y)$ is a prime ideal), since the dimension of an algebraic set is the maximum of its irreducible components  \cite[Proposition 2.8.5(i)]{BCR} - if $Y$ is reducible, replace it with an irreducible component of the same dimension. 
The restriction map induced by the inclusion $Y\subseteq X$ induces a map $\mathcal{P}(X) \to \mathcal{P}(Y)$ with kernel $\mathcal{I}(Y)/\mathcal{I}(X)$. Let $0=P_0\subset P_1 \subset P_2 \subset \cdots \subset P_{k-1}\subset P_k \subset \mathcal{P}(Y)$ be a chain of prime ideals (each inclusion is proper) of $\mathcal{P}(Y)$ of length $k=dim(Y)$ ($0$ is a prime ideal in $\mathcal{P}(Y)$ since $\mathcal{I}(Y)$ is a prime ideal). The preimages in $\mathcal{P}(X)$ will give a chain of prime ideals in $\mathcal{P}(X)$ which all contain $\mathcal{I}(Y)/\mathcal{I}(X)$ (the preimage of a prime ideal is prime). Then we can extend the chain by 1 (the $0$-ideal of $\mathcal{P}(X)$) to get a chain of length $k+1$ if $Y \neq X$. Thus $dim(X)\geq dim(Y)+1$.  
\end{proof}

\begin{lemma} \label{algebraic}
Let $X$ and $Y$ be real algebraic sets, with $Y\subseteq X$ and a surjective polynomial map $\varphi: Y\twoheadrightarrow X$. 
Then $Y=X$.
\end{lemma}
\begin{proof}

{\bf Remark: } One needs to be careful that $Y$ is an algebraic subset of $X$, not just the image of an injective algebraic map $Y \to X$ which may only have semi-algebraic image. An example of this sort of phenomenon (suggested to us by Nathan Dunfield) is the algebraic set $Y= \{xy=1\} \subset \mathbb{R}^2$. Projection to $X=\mathbb{R}$ gives the semi-algebraic set $\{x\neq 0\} \subset \mathbb{R}$. But this also has a surjective polynomial map to $\mathbb{R}$ by the polynomial map $(x,y) \mapsto x-y$. 
The argument below fails here because $\{x\neq 0\}$ is Zariski dense in $\mathbb{R}$, which cannot occur for a proper algebraic subset of an irreducible variety by the previous lemma. 

Let $X=X_1\cup \cdots \cup X_p$, where $X_i$  are the irreducible components of $X$ (\cite[Theorem 2.8.3(i)]{BCR}). 

For $i=1, \dots, p$, $\varphi^{-1}(X_i)$ is a Zariski closed subset of $Y$. 
Let $X_i'=X_i-\cup_{j\neq i} X_j$, then $clos_{Zar}(X_i')=X_i$  by the uniqueness of the irreducible decomposition.
Thus $dim(X_i)=dim(X_i') \leq dim(\varphi^{-1}(X_i'))$ by \cite[Proposition 2.8.2]{BCR} and  \cite[Theorem 2.8.8]{BCR}. Thus there is an irreducible component $U_i$ of $\varphi^{-1}(X_i)$ with $dim(U_i) = dim(\varphi^{-1}(X_i')) \geq dim(X_i)$ and $U_i\cap \varphi^{-1}(X_i') \neq \emptyset$. Hence we have $p$ irreducible algebraic subsets $U_1,\ldots, U_p\subset Y$ with $dim(U_i) \geq dim(X_i)$ and $U_i-\cup_{j\neq i} U_j \neq \emptyset$. 

But $U_i\subseteq X_j$ for some $j$. If $dim(X_i)=dim(X)$, then $dim(X_i)\leq dim(U_i)\leq dim(X_j)\leq dim(X)$, so $dim(U_i)=dim(X_j)$ and by Lemma \ref{dimension}, $U_i=X_j$. Thus $\varphi^{-1}(X_i) \supseteq X_j$, and the $X_j$'s with dimension $dim(X)$ are a permutation of the $X_i$'s with $dim(X_i)=dim(X)$. 

Suppose all irreducible components $X_i$ of $X$ with dimension $\geq k>0$ have $U_i=X_j$ for some $j$ with $dim(X_j)=dim(X_i)$, giving a bijection between the components of dimension $\geq k$. Let $dim(X_i)=k-1$. Then $U_i\subseteq X_j$ where $dim(X_j)=k-1$, since $X_j$ cannot have dimension $\geq k$ or else it would be in the preimage of some $X_m$ with $dim(X_m)\geq k$, contradicting that $\varphi(U_i) \cap X_i' \neq \emptyset$, where $X_i'$ is disjoint from $X_m$. Since $k-1=dim(X_i) \leq dim(U_i) \leq dim(X_j)=k-1$, we have $U_i=X_j$ by Lemma \ref{dimension}. Thus by reverse induction on dimension, we see that the $U_i$ are a permutation of the $X_j$, and hence $Y=X$.  
\end{proof}

\bibliography{Gordon_conjecture.bbl}
\end{document}